\documentclass[preprint]{elsarticle}

\makeatletter
\def\ps@pprintTitle{%
	\let\@oddhead\@empty
	\let\@evenhead\@empty
	\def\@oddfoot{\footnotesize\itshape
		{} \hfill\today}%
	\let\@evenfoot\@oddfoot
}
\makeatother

\usepackage{latexsym}
\usepackage{lineno,hyperref}
\usepackage{indentfirst}
\usepackage{amsxtra}
\usepackage{amssymb}
\usepackage{amsthm}
\usepackage{natbib}
\usepackage{amsmath}
\usepackage{tikz}
\usepackage{amscd}
\newtheorem{theor}{Theorem}
\newtheorem{prop}[theor]{Proposition}
\newtheorem{cor}[theor]{Corollary}

\newtheorem{lemma}[theor]{Lemma}

\theoremstyle{definition}

\newtheorem{ex}{Example}

\begin{document}

\begin{frontmatter}

\title{On the indecomposable involutive set-theoretic solutions of the Yang-Baxter equation of prime-power size}

\author{M. Castelli\corref{cor1}}
\ead{marco.castelli@unisalento.it}
%
%
\author{G. Pinto\corref{cor2}}
\ead{giuseppina.pinto@unisalento.it}

\author{W. Rump}
\ead{rump@mathematik.uni-stuttgart.de }
\cortext[cor1]{Corresponding author}
%
%
%
%
\begin{abstract}
We develop a method to construct all the indecomposable involutive set-theoretic solutions of the Yang-Baxter equation with a prime-power number of elements and cyclic permutation group. Moreover, we give a complete classification of the indecomposable ones having abelian permutation group and cardinality $pq$ (where $p$ and $q$ are prime numbers not necessarily distinct).
\end{abstract}
\begin{keyword}
\texttt{Cycle set\sep set-theoretic solution\sep Yang-Baxter equation}
\MSC[2010] 16T25\sep 20E22\sep 20F16
\end{keyword}

\end{frontmatter}

\section{Introduction}
A \emph{set-theoretic solution of the Yang-Baxter equation}, or shortly a \textit{solution}, is a pair $(X,r)$ where $X$ is a non-empty set and $r$ an invertible map from $X\times X$ to itself such that 
$$
r_1r_2r_1 = r_2r_1r_2 ,
$$
where $r_1:= r\times id_X$ and $r_2:= id_X\times r$.  Let $\lambda_x:X\to X$ and $\rho_y:X\to X$ be maps such that  
$$
r(x,y) = (\lambda_x(y), \rho_y(x))
$$ 
for all $x,y\in X$. A set-theoretic solution of the Yang-Baxter equation $(X, r)$ is said to be a left [\,right\,] non-degenerate if $\lambda_x\in Sym(X)$ [\,$\rho_x\in Sym(X)$\,] for every $x\in X$, and non-degenerate if it is left and right non-degenerate. \\ 
Drinfeld's paper \cite{drinfeld1992some} stimulated much interest in this subject. In recent years, after the seminal papers by Gateva-Ivanova and Van den Bergh \cite{gateva1998semigroups} and Etingof, Schedler and Soloviev \cite{etingof1998set} the involutive solutions $r$, i.e. $r^2=id_{X\times X}$, have received a lot of attention (see, for example, \cite{bachiller2015family,cacs2,etingof1998set,gateva2004combinatorial,gateva1998semigroups,rump2005decomposition, vendramin2016extensions}).\\
In \cite[Section 2.8]{etingof1998set}, Etingof, Schedler and Soloviev showed that every involutive non-degenerate \textit{decomposable} solution $(X,r)$ can be constructed starting from two smaller solutions $(Y,s)$ and $(Z,t)$, where $Y$ and $Z$ are proper subsets of $X$, $s=r_{|Y\times Y}$ and $t=r_{|Z\times Z}$. Recall that an involutive non-degenerate solution $(X,r)$ is said to be \textit{decomposable} if there exists a partition $\{Y,Z\}$ of $X$ such that $r\vert_{Y\times Y}$ and $r\vert_{Z\times Z}$ are non-degenerate solutions; otherwise, $(X,r)$ is called \textit{indecomposable}. Therefore, a possible strategy to find all the involutive non-degenerate solutions is the following: at first, construct all the indecomposable involutive non-degenerate solutions and then use the ideas developed by Etingof, Schedler and Soloviev to construct the involutive non-degenerate decomposable ones. In 1999 Etingof, Schedler and Soloviev \cite{etingof1998set} found, by computer calculations, all the indecomposable involutive non-degenerate solutions of cardinality at most eight; moreover, they showed that, up to isomorphism, the unique indecomposable involutive non-degenerate solution having a prime number $p$ of elements is given by $(\mathbb{Z}/p\mathbb{Z},u)$ and $u(x,y):=(y-1,x+1)$, for every $x,y\in \mathbb{Z}/p\mathbb{Z}$.\\
In recent years, some authors used the links between involutive solutions and other algebraic structures (see for example \cite{chouraqui2010garside,rump2007braces,gateva1998semigroups,cacs2,dehornoy2015set}) to provide new descriptions of the indecomposable ones: in particular, in 2010 Chouraqui \citep{chouraqui2010garside} give a characterization of indecomposable involutive non-degenerate solutions by Garside monoids, while A. Smoktunowicz and  A. Smoktunowicz \cite{smock} characterized these solutions by left braces. However, these interesting results do not allow to provide easily new families of examples.\\
In that regard, the first two authors together with F. Catino \cite{cacsp2018} developed an extension-theory of indecomposable involutive non-degenerate solutions that allow to construct concretely several new families of indecomposable involutive solutions. The main  tool used is the cycle set, an algebraic structure introduced by the third author \cite{rump2005decomposition}. Recall that a non-empty set $X$ with a binary operation $\cdot$ is a \emph{cycle set} if each left multiplication $\sigma_x:X\longrightarrow X,\; y\longmapsto x\cdot y$ is invertible and 
\begin{equation*}\label{cicloide}
(x\cdot y)\cdot (x\cdot z)= (y\cdot x)\cdot (y\cdot z), 
\end{equation*}
for all $x,y,z\in X$.
Moreover a cycle set $(X,\cdot)$ is \emph{non-degenerate} if the squaring map $\mathfrak{q}:X\longmapsto X,\; x\mapsto x\cdot x$ is bijective.
The third author \cite{rump2005decomposition} proved that if $(X,\cdot)$ is a non-degenerate cycle set, the map $r:X\times X\longmapsto X\times X$ defined by $r(x,y):=(\lambda_x(y), \rho_y(x))$, where $\lambda_x(y):=\sigma_x^{-1}(y)$ and $\rho_y(x):=\lambda_x(y)\cdot x$, is a non-degenerate involutive solution of the Yang-Baxter equation. Conversely, if $(X,r)$ is a non-degenerate involutive solution and $\cdotp$ the binary operation given by $x\cdot y:= \lambda_x^{-1}(y)$ for all $x,y\in X$, then $(X,\cdot)$ is a non-degenerate cycle set. The existence of this bijective correspondence allows us to move the study of involutive non-degenerate solutions to non-degenerate cycle sets, as already made in \cite{bon2019,JeP,cacsp2019,catino2015construction,cacsp2017,cacsp2018,cacsp2018quasi,rump2016quasi,vendramin2016extensions,lebed2017homology}.\\
The goal of the first part of this paper is to provide families of indecomposable involutive non-degenerate solutions by an approach, different from that used in \cite{cacsp2018}, for which the involutive Yang-Baxter Groups play a crucial role. Following \cite{cedo2010involutive}, a group $\mathcal{G}$ is said to be an \textit{involutive Yang-Baxter Group} if there exists an involutive non-degenerate solution $(X,r)$ having (associated) permutation group isomorphic to $\mathcal{G}$. Following the stategy suggested by Ced\'o, Jespers and Del R\'io \cite{cedo2010involutive} to classify finite involutive non-degenerate solutions (not necessarily indecomposable), one can determine the class of the involutive Yang-Baxter Groups and then describe all the involutive non-degenerate solutions having a fixed Yang-Baxter Group $\mathcal{G}$ as associated permutation group. As noted in \cite{cedo2010involutive}, the second step of this approach seems to be very hard. However, this became easier if we restrict to particular classes of indecomposable involutive non-degenerate solutions: in Section $3$ we give a method to construct all the indecomposable involutive non-degenerate solutions having associated permutation group isomorphic to $(\mathbb{Z}/p^k\mathbb{Z},+)$ (where $p$ is an arbitrary prime number and $k$ an arbitrary natural number). The use of this method will be a fundamental step to give a complete classification of indecomposable involutive non-degenerate solutions of size $p^2$ and cyclic permutation group. As we will see, these solutions are multipermutational of level at most $2$. An algorithm useful to construct all the finite involutive non-degenerate solutions (not necessarily indecomposable) of multipermutational level equal to $2$ has been developed recently in \cite{JePiZa}. However, this algorithm is quite technical and does not include any tools to distinguish the isomorphism classes. In the core of Section $4$ we will see that our method, which in full generality requires some technical hypothesis, can be easily used to construct (without computer calculations) all the indecomposable involutive non-degenerate solutions of size $p^2$ having cyclic permutation group and to distinguish the isomorphism classes.\\
In the last part of this paper we will consider indecomposable involutive non-degenerate solutions of size $pq$ (where $p,q$ are prime numbers not necessarily distinct) having abelian associated permutation group.  
In this context, we will show two key-results. In the first one, we will see that every indecomposable involutive non-degenerate solution of size $pq$ (with $p\neq q$) and with cyclic permutation group has multipermutational level equal to $1$. In the second one, we will prove that there is (up to isomorphisms) only one indecomposable involutive non-degenerate solution of size $p^2$ and with abelian non-cyclic permutation group. These results will allow us to give a complete classification of the indecomposable involutive non-degenerate solutions of size $pq$ and abelian permutation group.

\section{Some preliminary results}

A non-empty set $X$ with a binary operation $\cdot$ is a \emph{cycle set} if each left multiplication $\sigma_x:X\longrightarrow X,\; y\longmapsto x\cdot y$ is invertible and 
\begin{equation}\label{cicloide}
(x\cdot y)\cdot (x\cdot z)= (y\cdot x)\cdot (y\cdot z),
\end{equation}
for all $x,y,z\in X$. 
A cycle set $(X,\cdot)$ is \emph{non-degenerate} if the squaring map $\mathfrak{q}:X\longmapsto X,\; x\mapsto x\cdot x$ is bijective, and it is  \emph{square-free} if $\mathfrak{q}=id_X$. It is known that every finite cycle set is non-degenerate (see \cite[Theorem 2]{rump2005decomposition}).
Moreover, the image $\sigma(X)$ of the map $\sigma:X\longrightarrow Sym(X),\; x\mapsto \sigma_x$ can be endowed with an induced binary operation 
$$ 
\sigma_x\cdot\sigma_y:=\sigma_{x\cdot y}
$$
which satisfies (\ref{cicloide}). The third author \cite{rump2005decomposition} showed that $(\sigma(X),\cdot)$ is a (non-degenerate) cycle set if and only if $(X,\cdot)$ is non-degenerate. The cycle set $\sigma(X)$ is called the \emph{retraction} of $(X,\cdot)$.\\
Moreover, the cycle set $(X,\cdot)$ is said to be \emph{irretractable} if $|X|>1$ and $(\sigma(X),\cdot)$ is isomorphic to $(X,\cdot)$, otherwise it is called \emph{retractable}.\\
A non-degenerate cycle set $(X,\cdot)$ is called \emph{multipermutational of level $m$}, if $m$ is the minimal non-negative integer such that $\sigma^m(X)$ has cardinality one, where 
$$
\sigma^0(X):= X \quad \textrm{and}\quad \sigma^n(X):=\sigma(\sigma^{n-1}(X)),\quad \textrm{for}\quad n\geq 1 .
$$
In this case we write $mpl(X)=m$. Obviously a multipermutational cycle set is retractable but the converse is not necessarly true.\\
The permutation group $\mathcal{G}(X)$ of $X$ is the subgroup of $Sym(X)$ generated by the image $\sigma(X)$ of $\sigma$. A subset $Y$ of $X$ is called \emph{$\mathcal{G}(X)$-invariant} when $\sigma_x(Y)\subseteq Y$ for all $x\in X$. A cycle set $X$ is said to be \emph{decomposable} if there exists a partition $\{Y,Z\}$ of $X$ such that $Y$ and $Z$ are $\mathcal{G}(X)$-invariant subsets, otherwise it is called \emph{indecomposable}. Referring to a well known result (see \cite[Prop. 2.12]{etingof1998set}) in terms of  cycle sets, we know that a cycle set $X$ is indecomposable if and only if the permutation group $\mathcal{G}(X)$ acts transitively on $X$.\\
\\
Following Vendramin \cite{vendramin2016extensions}, if $I$ is a cycle set and $S$ a non-empty set, then   $\alpha:I\times I\times S\longrightarrow Sym(S)$, $(i,j,s)\mapsto \alpha_{(i,j)}(s,-)$ is a \emph{dynamical cocycle} if 
\begin{equation}
\alpha_{(i\cdot j,i\cdot k)}(\alpha_{(i,j)}(r,s),\alpha_{(i,k)}(r,t))=\alpha_{(j\cdot i,j\cdot k)}(\alpha_{(j,i)}(s,r),\alpha_{(j,k)}(s,t)),
\end{equation}
for all $i,j,k\in I$, $r,s,t\in S$. Moreover, if $\alpha$ is a dynamical cocycle, then the cycle set $S\times_{\alpha} I:=(S\times I,\cdot)$, where
\begin{equation}
(s,i)\cdot (t,j):=(\alpha_{(i,j)}(s,t),i\cdot j),
\end{equation}
for all $i,j\in I$, $s,t\in S$, is called \emph{dynamical extension} of $I$ by $\alpha$. \\
By \cite[Theorem 3.8]{etingof1998set} (or by \cite[Proposition 2]{cacsp2018} applied with $Y=\sigma(X)$), if $X$ is an indecomposable finite cycle set, then there exist a set $S$ and a dynamical cocycle $\alpha:\sigma(X)\times \sigma(X)\times S\longrightarrow Sym(S)$ such that $X\cong \sigma(X)\times_{\alpha} S$.\\
\\
From now on, we will focus our attention on indecomposable finite cycle sets having abelian permutation group. In the rest of this section we will show that the cycle sets of this type are always non-degenerate and multipermutational: note that by \cite[Theorem 2]{rump2005decomposition} and \cite{cedo2014braces}, these facts are true for every finite cycle sets with abelian permutation group, but, for the indecomposable ones, we are able to exhibit two easier proofs.

\begin{prop}\cite{rumpcourse}\label{dimoabel}
Let $X$ be an indecomposable finite cycle set such that $|X|>1$ and with abelian permutation group $\mathcal{G}(X)$. Then $\sigma(X)$ is an indecomposable cycle set and $X$ is multipermutational.
\end{prop}

\begin{proof} 
Since $X$ is finite, it is easy to check that $\sigma(X)$ is an indecomposable cycle set. Moreover, it is easy to show that $\mathcal{G}(\sigma(X))$ is abelian. So it is sufficient to show that $X$ is retractable: indeed, in this case, by induction on $|X|$ we obtain easily that $X$ is multipermutational. So we prove that $|\sigma(X)|<|X|$. 
If $\sigma(X)=X$ then all the elements of $X$ acts differently to each other and, since $\mathcal{G}(X)$ is a transitive abelian group on $X$, we have that $|\mathcal{G}(X)|=|X|$ and therefore there exists a unique $x\in X$ such that $\sigma_x=id_X$. If $y\in X$ then
$$ y\cdotp z=(x\cdotp y)\cdotp(x\cdotp z)=(y\cdotp x)\cdotp (y\cdotp z),$$
for all $z\in X$. Since $\mathcal{G}(X)$ is a transitive abelian group on $X$, it is regular. Then $\sigma_{(y\cdotp x)}=\sigma_x$ for all $y\in X$. Now, since $\sigma(X)=X $ it follows that $y\cdotp x=x$ for all $y\in X$ and hence $\mathcal{G}(X)$ does not act transitively on $X$, in contradiction with the indecomposability of $X$.
\end{proof}

\begin{prop}\cite{rumpcourse}\label{abeldim}
Let $X$ be an indecomposable finite cycle set with abelian permutation group $\mathcal{G}(X)$. Then $X$ is non-degenerate.
\end{prop}

\begin{proof}
We prove the thesis by induction on $|X|$. If $|X|=2$ the thesis is trivial. Now, let $X$ be a cycle set with $|X|>2$ and $x,y\in X$ such that $x\cdotp x=y\cdotp y$. By the previous proposition $|\sigma(X)|<|X|$ and, by inductive hypothesis, the equality $\sigma_{x\cdotp x}=\sigma_{y\cdotp y}$ implies $\sigma_x=\sigma_y$. Hence 
$$x\cdotp x=y\cdotp y=x\cdotp y$$
 that implies $x=y$. This proves the injectivity of the squaring map $\mathfrak{q}$. By the finiteness of $X$ the thesis follows. 
\end{proof}

\section{Indecomposable cycle sets of order $p^k$ and cyclic permutation group} 

In this section we give a method to construct all the indecomposable cycle sets of cardinality a prime-power and cyclic permutation group. For this purpose, the following lemma will be of crucial importance. 

\begin{lemma}\label{retrgen}
Let $X$ be an indecomposable finite cycle set having cyclic permutation group $\mathcal{G}(X)$ generated by an element $\alpha\in Sym(X)$.\\
Then $\sigma_{\alpha^{|\sigma(X)|}(x)}=\sigma_x$ for every $x\in X$. In particular $|\sigma(X)|$ is the minimal integer such that the previous equality holds.
\end{lemma}

\begin{proof}
Let $x\in X$ and $k\in \mathbb{N}$ such that $\sigma_{\alpha^{k}(x)}=\sigma_x$. Now we show that $ \sigma_{\alpha^{k+j}(x)}=\sigma_{\alpha^j(x)}$ for every $j\in \mathbb{N}$.
Let $y_1, \ldots ,y_s\in X$ such that $\alpha=\sigma_{y_1}\cdots \sigma_{y_s}$. Then we have that
\begin{align*}
\sigma_{\alpha^{k+1}(x)}&=\sigma_{y_1\cdot(y_2\cdot(\cdots (y_s\cdot \alpha^k(x))\cdots)))}\\
&= \sigma_{y_1}\cdot (\cdots (\sigma_{y_s}\cdot \sigma_{\alpha^k(x)})\\
&= \sigma_{y_1}\cdot (\cdots (\sigma_{y_s}\cdot \sigma_{x})\\
&=\sigma_{\alpha(x)}
\end{align*}
hence the thesis follows for $j=1$. Now, 
\begin{align*}
\sigma_{\alpha^{k+j}(x)}&=\sigma_{y_1\cdot(y_2\cdot(\cdots (y_s\cdot \alpha^{k+j-1}(x))\cdots)))}\\
&= \sigma_{y_1}\cdot (\cdots (\sigma_{y_s}\cdot \sigma_{\alpha^{k+j-1}(x)})\\
&= \sigma_{y_1}\cdot (\cdots (\sigma_{y_s}\cdot \sigma_{\alpha^{j-1}(x)})\\
&=\sigma_{\alpha^j(x)}
\end{align*}
where in the third equality we used the inductive hypothesis. This implies that $\bar{n}:=min\{ n\mid n\in \mathbb{N}\; \sigma_{\alpha^n(x)}=\sigma_x  \}$ does not depend on the choice of the element $x$ and that $\sigma(X)=\{\sigma_x,\sigma_{\alpha(x)},...,\sigma_{\alpha^{\bar{n}-1}(x)}\}$, hence the thesis.
\end{proof}

An other simple but useful result for our scope is the following lemma, in which we prove that if $X$ is an indecomposable finite cycle set of cardinality a prime-power and cyclic permutation group, then $\mathcal{G}(X)$ is generated by one left multiplication. Since indecomposable cycle sets of prime size have been completely described in \cite{etingof1998set}, from now on, if not specified, when a cycle set has size $p^k$ for some prime number $p$, then $k$ will be a natural number greater than $1$.

\begin{lemma}\label{ungen}
Let $X$ be an indecomposable cycle set of order $p^k$ and with cyclic permutation group $\mathcal{G}(X):=<\psi>$, where $p$ is a prime integer and $k$ a natural number. Then there exists $x\in X$ such that $<\sigma_x>=\mathcal{G}(X)$.
\end{lemma}

\begin{proof}
By \cite[Prop. 2.12]{etingof1998set}, $\mathcal{G}(X)$ acts transitively on $X$, therefore we have that $|\mathcal{G}(X)|=|X|=p^k$. Moreover, for every $x\in X$ there exists $a_x\in \mathbb{N}$ such that $\sigma_x=\psi^{a_x}$, hence it is sufficient to show that there exists $x\in X$ such that $(a_x,p)=1$.\\
If we suppose $p\mid a_x$ for every $x\in X$, then $\mathcal{G}(X)$ is contained in $<\psi^p>$ and hence $|\mathcal{G}(X)|<p^k$, a contradiction.
\end{proof}

Referring to Lemma \ref{ungen}, if $X$ is an indecomposable cycle set having $p^k$ elements and cyclic permutation group, hereinafter, without loss of generality, we can suppose that $X:=\{0,...,p^k-1\}$,  $\mathcal{G}(X)=<t_1>$ and $\sigma_0=t_1$, where $t_1$ is the permutation given by $t_1:=(0,...,p^k-1)$.\\
In order to show the main results of this section, we exhibit some other preliminary lemmas.

\begin{lemma}\label{congru}
Let $X:=\{0,...,p^k-1\}$ be an indecomposable cycle set with cyclic permutation group $\mathcal{G}(X):=<t_1>$, where $p$ is a prime integer and $k$ a natural number, $j_0,...,j_{p^k-1}\in \mathbb{N}$ such that $\sigma_{i}=t_{1}^{j_i}$ for every $i=0,...,p^k-1$. Then $j_i\equiv j_{i+|\sigma^s(X)|}\;(mod\; |\sigma^{s-1}(X)|)$ for every $i\in \{0,...,p^k-1\}$ and $s\in \mathbb{N}$.
\end{lemma}

\begin{proof}
Since by Proposition \ref{dimoabel} $X$ is multipermutational, we prove the thesis by induction on $mpl(X)$. If $mpl(X)=1$ the thesis is trivial. Now, let $\bar{t}_1$ be the element of $Sym(\sigma(X))$ given by $\bar{t}_1=(\sigma_{x_0}\;...\;\sigma_{x_{|\sigma(X)|-1}}) $. Then it follows that $\mathcal{G}(\sigma(X))=<\bar{t}_1>$ and $\sigma_{\sigma_{x_i}}=\bar{t}_1^{j'_i}$ for every $i=0,...,|\sigma(X)|-1$, where $j'_i$ is the rest of $j_i$ by the division for $|\sigma(X)|$. By inductive hypothesis we have that $j'_i\equiv j'_{i+|\sigma^{s+1}(X)|}\;(mod\; |\sigma^{s}(X)|)$ for every $s\in \mathbb{N}$ and, since by \cite[Lemma 1]{cacsp2018} $|\sigma^s(X)|$ divides $|\sigma(X)|$ for every $s\in \mathbb{N}$, it holds that $j_i\equiv j_{i+|\sigma^{s+1}(X)|}\;(mod\; |\sigma^{s}(X)|)$ for every $i\in \{0,...,p^k-1\}$ and $s\in \mathbb{N}$.\\
It remains to prove that $j_i\equiv j_{i+|\sigma(X)|}\;(mod\; |X|)$ for every $i\in \{0,...,p^k-1\}$, but this fact holds by Lemma \ref{retrgen}.
\end{proof}

\begin{cor}\label{consegen}
Let $X:=\{0,...,p^k-1\}$ be an indecomposable cycle set with cyclic permutation group $\mathcal{G}(X):=<t_1>$, where $p$ is a prime integer and $k$ a natural number, $j_0,...,j_{p^k-1}\in \mathbb{N}$ such that $\sigma_{i}=t_{1}^{j_i}$ for every $i=0,...,p^k-1$. Then $<\sigma_x>=\mathcal{G}(X)$ for every $x\in X$.
\end{cor}

\begin{proof}
By Proposition \ref{dimoabel}, $X$ is a multipermutational cycle set. Let $s$ be its multipermutational level. By the previous proposition $j_{i+1}\equiv j_i\;(mod\; \sigma^{s-1}(X))$ for every $i\in \{0,...,p^k-1\}$. Since $j_0=1$ and $p||\sigma^{s-1}(X)|$, it follows that $j_i\equiv 1\;(mod\;p)$ for every $i\in \{0,...,p^k-1\}$, hence the thesis.
\end{proof}

\begin{lemma}\label{scomp}
Let $m\in \mathbb{N}$, $p$ a prime number, $j_0,...,j_{m+1}\in \mathbb{N}$ such that $j_0=0$ and $j_i<j_{i+1}$ for every $i\in \{0,...,m\}$ and $n\in \{0,...,p^{j_{m+1}}-1\}$. Then there exists a unique $(a_0,...,a_m)\in \mathbb{N}^{m+1}$ such that $n=a_0+a_1p^{j_1}+...+a_{m}p^{j_m}$, where $a_i\in \{0,...,\frac{p^{j_{i+1}}}{p^{j_i}}-1\}$ for every $i\in \{0,...,m\}$.
\end{lemma}

\begin{proof}
We prove the thesis by induction on $m$. If $m=1$ the thesis is trivial. Therefore, suppose that for every $n\in \{0,...,p^{j_{m}}-1\}$ there exists a unique $(a_0,...,a_{m-1})\in \mathbb{N}^{m}$ such that $n=a_0+a_1p^{j_1}+...+a_{m-1}p^{j_{m-1}}$, where $a_i\in \{0,...,\frac{p^{j_{i+1}}}{p^{j_i}}-1\}$ for every $i\in \{0,...,m-1\}$.\\
Now, let $n\in \{0,...,p^{j_{m+1}}-1\} $ and consider two cases. If $n<p^{j_m}$ then, by inductive hypothesis, there exists a unique $(a_0,...,a_{m-1})\in \mathbb{N}^{m}$ such that $n=a_0+a_1p^{j_1}+...+a_{m-1}p^{j_{m-1}}$, hence  $(a_0,...,a_{m-1},0)$ is an element of $\mathbb{N}^{m+1}$ such that $n=a_0+a_1p^{j_1}+...+a_{m-1}p^{j_{m-1}}+0p^{j_m}$. The unicity of $(a_0,...,a_{m-1},0)$ follows from the fact that, if $a_m>0$, then $n>p^{j_m}-1$.\\
If $n>p^{j_m}-1$ then there exist  $a\in \{0,...,\frac{p^{j_{m+1}}}{p^{j_m}}-1\}$ and $b\in \{0,...,p^{j_{m}}-1 \}$ such that $n=ap^{j_m}+b$. Moreover, such pair $(a,b)$ is unique. Then, by inductive hypothesis applied on $b$, we obtain that there exist a unique element $(b_0,...,b_m)\in \mathbb{N}^{m+1}$ such that $n=b_0+b_1p^{j_1}+...+b_{m}p^{j_m}$.\\
\end{proof}

Now we are able to exhibit the main results. At first, we show a method to construct indecomposable cycle sets of prime-power size and cyclic permutation group. Since the multipermutational cycle sets of level $1$ are in some sense "trivial", in this context we will consider cycle sets of level at least $2$. 

\begin{theor}\label{costruz2}
Let $X:=\{0,...,p^k-1\}$, $n\in \mathbb{N}\setminus\{1\}$, $j_0,...,j_{n}\in \mathbb{N}\cup\{0\}$ such that $j_{n}=0$, $j_0=k$, $j_i<j_{i-1}$ for every $i\in \{1,...,n\}$ and $\{f_i\}_{i\in \{1,...,n-1\}}$ a set of functions such that
$$f_{i}:\mathbb{Z}/p^{j_i}\mathbb{Z}\longrightarrow \{0,...,p^{j_{i-1}}/p^{j_i}-1\}$$
$f_{i}(0)=0$ for every $i\in \{1,...,n-1 \}$ and the function
$$\varphi_{i}:\{0,...,p^{j_{i}}-1\}\longrightarrow \{1,...,p^{j_{i-1}}-1\}$$
$$l\mapsto 1+p^{j_{n-1}}f_{n-1}(l)+...+p^{j_{i}}f_{i}(l)$$
is injective for every $i\in \{1,...,n-1\}$. Moreover, set $\psi:=(0\;...\;p^k-1)$ and
 \begin{equation}
\sigma_{i}:=\psi^{1+p^{j_{n-1}}f_{n-1}(i)+...+p^{j_1}f_{1}(i)}
\end{equation}
for every $i\in X$. Define 
$$K_{j,i}:=j+1+p^{n-1}f_{n-1}(i)+...+p^{j_2}f_2(i)$$
and
$$Q_{j,i}:= p^{j_{n-1}}f_{n-1}(i)+...+p^{j_1}f_1(i)+p^{j_{n-1}}f_{n-1}(K_{j,i})+...+p^{j_1}f_1(K_{j,i})$$
for every $i,j\in X$ and suppose that $Q_{i,j}\equiv Q_{j,i}\;(mod\;p^k)$ for every $i,j\in \{0,...,p^k-1\}$.\\
Then $X$ is an indecomposable cycle set of level $n$ and cyclic permutation group $<\psi>$ such that $|\sigma^i(X)|=p^{j_i}$ for every $i\in\{0,...,n\}$.

\end{theor}

\begin{proof}
Clearly every left multiplication is bijective. Moreover, since $Q_{i,j}\equiv Q_{j,i}\;(mod\;p^k)$ for every $i,j\in X$ we obtain that
$$\sigma_{\sigma_{i}(j)}\sigma_{i}=t_1^{2+Q_{i,j}}=t_1^{2+Q_{j,i}}=\sigma_{\sigma_{j}(i)}\sigma_{j}$$
for every $i,j\in X$, hence $X$ is a cycle set. By the definition of the left multiplications it follows that $\mathcal{G}(X)=<\psi>$ and $X$ is indecomposable.\\
Now we show  that $|\sigma^i(X)|=p^{j_i}$ for every $i\in\{0,...,n\}$ by induction on $i$. If $i=1$ then, since $\varphi_1$ is injective, we obtain that $\sigma_x=\sigma_y$ if and only if $x\equiv y\;(mod\;p^{j_1})$, hence $|\sigma(X)|=p^{j_1}$. Hence, suppose that $|\sigma^{i}(X)|$ has cardinality $p^{j_{i}}$. With a standard calculation one can see that $|\sigma^{i}(X)|$ is isomorphic to the cycle set on $\{0,...,p^{j_i}-1\}$ given by
$$\sigma_j:=\begin{cases} 
\bar{t}_1^{1+|\sigma^{n-1}(X)|f_{n-1}(j)+...+|\sigma^{i+1}(X)|f_{i+1}(j)} & \mbox{if } i<n-1\\ 
\bar{t}_1 & \mbox{if } i=n-1\end{cases} $$
for every $j\in \{0,...,p^{j_i}-1 \}$,where $\bar{t}_1:=(0,...,p^{j_i}-1)$.\\
If $i<n-1$, since $\varphi_{i+1}$ is injective, we obtain that $\sigma_j=\sigma_{j'}$ if and only if $j\equiv j'\;(mod\;p^{j_{i+1}})$, therefore $|\sigma^{i+1}(X)|=p^{j_{i+1}}$. If $i=n-1$ then all the left multiplications of $\sigma^{i}(X)$ are equal to $\bar{t}_1$, hence $|\sigma^{i+1}(X)|=|\sigma^{n}(X)|=1=p^{j_n}$. 
\end{proof}

The previous theorem allows us to obtain all the indecomposable cycle sets of prime-power size and cyclic permutation group, as we can see in the following theorem. 

\begin{theor}\label{costruz}
Let $X:=\{0,...,p^k-1\}$ be an indecomposable cycle set of order $p^k$, where $p$ is a prime integer and $k$ a natural number, having multipermutational level $n>1$ and with cyclic permutation group $\mathcal{G}(X)=<t_1>$, where $\sigma_0=t_1$. Then for every $j\in \{1,...,n-1\}$ there exists a function 
$$f_j:\mathbb{Z}/|\sigma^j(X)|\mathbb{Z}\longrightarrow \{0,...,|\sigma^{j-1}(X)|/|\sigma^j(X)|-1\}$$
such that $f_j(0)=0$, the function 
$$\varphi_b:\{0,...,|\sigma^b(X)|-1\}\longrightarrow \{1,...,|\sigma^{b-1}(X)|-1\}$$
$$i\mapsto 1+|\sigma^{n-1}(X)|f_{n-1}(i)+...+|\sigma^{b}(X)|f_{b}(i)$$
 is injective and
\begin{equation}\label{req1}
\sigma_{i}=t_1^{1+|\sigma^{n-1}(X)|f_{n-1}(i)+...+|\sigma(X)|f_{1}(i)}
\end{equation}
for every $i\in \{0,...,p^k-1\}$ and $b\in \{1,...,n-1\}$. Moreover, if we put 
$$K_{j,i}:=j+1+|\sigma^{n-1}(X)|f_{n-1}(i)+...+|\sigma^2(X)|f_2(i)$$
and
$$Q_{j,i}:= |\sigma^{n-1}(X)|f_{n-1}(i)+...+|\sigma(X)|f_1(i)+|\sigma^{n-1}(X)|f_{n-1}(K_{j,i})+...+|\sigma(X)|f_1(K_{j,i})$$
for every $i,j\in \{0,...,p^k-1\}$ then $Q_{i,j}\equiv Q_{j,i}\;(mod\;p^k)$.\\
\end{theor}

\begin{proof}
By Corollary \ref{consegen} for every $j\in X$ there exists $m_j\in \{1,...,p^k-1\}$ such that 
$$\sigma_{j}=t_1^{m_j}\; and \; m_j\equiv 1\;(mod\;p).$$
By Lemma \ref{scomp}, there exist a unique $n-1-$uple $(a_1,...,a_{n-1})$ such that
$$m_j=1+a_{1}|\sigma(X)|+...+a_{n-1}|\sigma^{n-1}(X)| $$
and $a_j\in \{0,...,|\sigma^{j-1}(X)|/|\sigma^j(X)|-1\}$ for every $j\in \{1,...,n-1\}$.\\
Now, by Lemma \ref{congru} we can define the function  
$$f_l:\mathbb{Z}/|\sigma^l(X)|\mathbb{Z}\longrightarrow \{0,...,|\sigma^{l-1}(X)|/|\sigma^l(X)|-1\}$$
$$f_l([j]):=a_l$$
for every $j\in X$ and $l\in \{1,...,n-1\}$, where $[j]$ denote the class of $j$ module $|\sigma^l(X)|$. By construction we have that (\ref{req1}) holds.\\
Moreover, since $\sigma_{\sigma_{i}(j)}\sigma_{i}=\sigma_{\sigma_{j}(i)}\sigma_{j}$ for every $i,j\in X$, we have that $t_1^{2+Q_{i,j}}=t_1^{2+Q_{j,i}}$, hence $Q_{i,j}\equiv Q_{j,i}\;(mod\;p^k)$.\\
It remains to show that the function $\varphi_b$ is injective for every $b\in \{1,...,n-1\}$. By Lemma \ref{retrgen}, we have that $\varphi_1$ is injective.\\
Now, $\sigma(X)$ is an indecomposable cycle set such that $\mathcal{G}(\sigma(X))=<\bar{t}_1>$, where $\bar{t}_1:=(\sigma_{0},...,\sigma_{{|\sigma(X)|}})$, and $\sigma_{\sigma_i}=\bar{t}_1^{1+|\sigma^{n-1}(X)|f_{n-1}(i)+...+|\sigma^2(X)|f_{2}(i)}$ for every $\sigma_i\in \sigma(X)$. By Lemma \ref{retrgen} applied to $\sigma(X)$ we obtain that $\varphi_2$ is injective. With a similar argument on $\sigma^z(X)$ one can show that $\varphi_{z+1}$ is injective for every $z\in\{1,...,n-2\}$.
\end{proof}

As an easy consequence of the previous theorems, we obtain the following corollary.

\begin{cor}
Let $X$ be a finite indecomposable cycle set having permutation group $\mathcal{G}(X)$ isomorphic to $(\mathbb{Z}/p^k\mathbb{Z},+)$. Then $X$ is either a multipermutational cycle set of level $1$ or one of the cycle sets constructed in Theorem \ref{costruz2}.
\end{cor}

\begin{proof}
Since $X$ is indecomposable and $\mathcal{G}(X)$ is cyclic, we have that $X$ is multipermutational and $|X|=|\mathcal{G}(X)|=p^k$. Then the thesis follows by Lemma \ref{ungen} and Theorem \ref{costruz}.
\end{proof}

We conclude this section using the previous results to provide some examples of indecomposable cycle sets with cyclic permutation group.

\begin{ex}
Let $X:=\{0,1,2,3\}$, $n:=2$, $j_0:=4,$ $j_1:=2$, $j_2:=1$ and $f_1:\mathbb{Z}/2\mathbb{Z}\longrightarrow \{0,1\}$ the function given by $f_1(0):=0$ and $f_1(1):=1$. Moreover, let $\varphi_1,\{\sigma_x\}_{x\in X},$ $\{K_{i,j}\}_{i,j\in X}$ and $\{Q_{i,j}\}_{i,j\in X}$ be as in Theorem \ref{costruz2}.\\
Then, by Theorem \ref{costruz2} $X$ is an indecomposable cycle set of level $2$ and with permutation group generated by the $4$-cycle $(0\;1\;2\;3)$. In particular,
$$\sigma_0=\sigma_2:=(0\;1\;2\;3)\quad  \sigma_1=\sigma_3:=(0\;3\;2\;1).$$

\end{ex}

\begin{ex}
Let $X:=\{0,1,2,3,4,5,6,7\}$, $n:=2$, $j_0:=8,$ $j_1:=2$, $j_2:=1$ and $f_1:\mathbb{Z}/2\mathbb{Z}\longrightarrow \{0,1,2,3\}$ the function given by $f_1(0):=0$ and $f_1(1):=2$. Moreover, let $\varphi_1,\{\sigma_x\}_{x\in X},$ $\{K_{i,j}\}_{i,j\in X}$ and $\{Q_{i,j}\}_{i,j\in X}$ be as in Theorem \ref{costruz2}.\\
Then, by Theorem \ref{costruz2} $X$ is an indecomposable cycle set of level $2$ and with permutation group generated by the $8$-cycle $(0\;1\;2\;3\;4\;5\;6\;7)$. In particular,
$$\sigma_0=\sigma_2=\sigma_4=\sigma_6:=(0\;1\;2\;3\;4\;5\;6\;7)$$
$$ \sigma_1=\sigma_3=\sigma_5=\sigma_7:=(0\;5\;2\;7\;4\;1\;6\;3).$$

\end{ex}

\begin{ex}
Let $X:=\{0,1,2,...,30,31\}$, $n:=3$, $j_0:=32,$ $j_1:=8$, $j_2:=2$, $j_3=1$, $f_1:\mathbb{Z}/8\mathbb{Z}\longrightarrow \{0,1,2,3\}$ the function given by 
$$f_1(0)=f_1(1):=0\quad f_1(2)=f_1(3):=3$$
$$f_1(4)=f_1(5):=2\quad f_1(6)=f_1(7):=1$$ 
and $f_2:\mathbb{Z}/2\mathbb{Z}\longrightarrow \{0,1,2,3\}$ the function given by $f_2(0):=0$ and $f_2(1)=2$.
Moreover, let $\varphi_1,\varphi_2,\{\sigma_x\}_{x\in X},$ $\{K_{i,j}\}_{i,j\in X}$ and $\{Q_{i,j}\}_{i,j\in X}$ be as in Theorem \ref{costruz2}.\\
Then, by Theorem \ref{costruz2} $X$ is an indecomposable cycle set of level $3$ and with permutation group generated by $\alpha:=(0\;1\;...\;30\;31)$. The left multiplications of $X$ are defined as follows 
$$\sigma_x:= \begin{cases} 
\alpha &\mbox{if } x \equiv 0\;(mod\; 8) \\ 
\alpha^5 & \mbox{if } x \equiv 1\;(mod\; 8)\\ 
\alpha^{25} & \mbox{if } x \equiv 2\;(mod\; 8)\\ 
\alpha^{29} & \mbox{if } x \equiv 3\;(mod\; 8)\\ 
\alpha^{17} & \mbox{if } x \equiv 4\;(mod\; 8)\\ 
\alpha^{21} & \mbox{if } x \equiv 5\;(mod\; 8)\\ 
\alpha^9 & \mbox{if } x \equiv 6\;(mod\; 8)\\ 
\alpha^{13} & \mbox{if } x \equiv 7\;(mod\; 8)\end{cases} $$
for every $x\in X$. In particular, $\sigma(X)$ is isomorphic to the previous example.

\end{ex}

\section{Indecomposable cycle sets of cardinality pq and cyclic permutation group} 

In the first part of this section we use the results of the previous one to give a complete classification of indecomposable cycle sets of size $p^2$ (where $p$ is a prime number) and with cyclic permutation group.\\
In the end, we will show that all the indecomposable cycle sets of cardinality $pq$ (where $p$ and $q$ are distinct prime numbers) and cyclic permutation group have multipermutational level equal to $1$.
\newline
We start the section with a simple but useful lemma.

\begin{lemma} \label{lemma2}
Let $f:\mathbb{Z}/p\mathbb{Z}\to \{0,...,p-1\}$ be a bijection such that $f(0)=0$ and
\begin{equation} \label{eq}
f(i+1)+f(j)\equiv f(i)+f(j+1)(mod\; p)
\end{equation} 
for all $i,j\in \mathbb{Z}/p\mathbb{Z}$.
Then there exist $t\in \mathbb{Z}/p\mathbb{Z}$, with $t\neq 0$ such that 
\begin{equation}\label{eq2}
f(k)=\overline{kt}
\end{equation}
for every $k\in \mathbb{Z}/p\mathbb{Z}$, where $\overline{kt}$ is the representative of $kt$ in $\{0,...,p-1\}$.
Conversely, every maps of type \eqref{eq2} satisfies \eqref{eq}. In particular there are $(p-1)$ bijections  of this type.
\end{lemma}

\begin{proof}
It is a straightforward calculation.
\end{proof}

\begin{cor}\label{prop3}
Let $X:=\{0,...,p^2-1\}$ be an indecomposable cycle sets of cardinality $p^2$ of multipermutational level equal to $2$ having cyclic permutation group. Then there exists a bijection $f:\mathbb{Z}/p\mathbb{Z}\to \{0,...,p-1\}$ satisfying \eqref{eq} such that $f(0)=0$ and $\sigma_{i}=t_1^{1+pf([i])}$ for every $i=\{0, \ldots p^2-1\}$, where $[i]$ is the class of $i$ module $p$.\\
Conversely, let $X:=\{x_0, \ldots , x_{p^2-1}\}$, $f:\mathbb{Z}/p\mathbb{Z}\to \{0,...,p-1\}$ a bijective map with $f(0)=0$ such that \eqref{eq} holds and $\phi$ the permutation given by $\phi:=(x_0,\ldots x_{p^2-1})$. Let $\cdot$ be the binary operation  on $X$ given by
$\sigma_{x_i}:=\phi^{1+p f([i])}$ for every $i\in \{0,\ldots ,p^{2}-1\}$. Then $(X,\cdot)$ is an indecomposable cycle set of multipermutational level equal to $2$ and with $\mathcal{G}(X)=<\phi>$.
\end{cor}

\begin{proof}
It follows by Theorems \ref{costruz2} and \ref{costruz}.
\end{proof}

By the previous Corollary, we have that for every indecomposable cycle set $X$ of size $p^2$, multipermutational level $2$ and with cyclic permutation group $\mathcal{G}(X)=<\sigma_x>$ (where $x$ is an arbitrary element of $X$) there exists a bijective function $f:\mathbb{Z}/p\mathbb{Z}\to \{0,...,p-1\}$ such that $f(0)=0$ and $\sigma_y:=\sigma_x^{1+pf([i])}$ for every $y\in X$, where $i$ is a natural number such that $y=\sigma_x^i(x)$. It is easy to see that the function $f$ does not depend on the choice of the element $x$. In the following proposition we will see that $f$ allows us to distinguish the isomorphism classes of these cycle sets.

\begin{prop}\label{propoiso}
Let $X$ and $Y$ be indecomposable cycle sets of cardinality $p^2$, multipermutational level $2$ and with cyclic permution groups. Let $x_0\in X$ such that $\sigma_{x_0}=\phi=(x_0...x_{p^2-1})$. Let $f_X:\mathbb{Z}/p\mathbb{Z}\to \{0,...,p-1\}$ and $f_Y:\mathbb{Z}/p\mathbb{Z}\to \{0,...,p-1\}$ be the bijective functions constructed as in the previous remark, hence $\sigma_{x_i}=\phi^{1+pf_X([i])}$ for every $i\in \{0,...,p-1\}$. Then $f_X=f_Y$ if and only if $X\cong Y$.
\end{prop}

\begin{proof}
Suppose that there exist an isomorphism $F$ from $X$ to $Y$ and define $y_0:= F(x_0)$ and $y_i:=y_0\cdotp y_{i-1}$ for every $i>0$. Then 
$$F(x_i\cdotp x_j)=F(\phi^{1+pf_X(i)}(x_j))=F(\sigma_{x_0}^{1+pf_X([i])}(x_j))=\sigma_{y_0}^{1+pf_X([i])}(y_j)=y_{j+1+pf_X([i])} $$
and
$$F(x_i)\cdotp F(x_j)=y_i\cdotp y_j=\sigma_{y_0}^{1+pf_Y([i])}(y_j)=y_{j+1+pf_Y([i])}$$
for every $i,j\in \{0,...,p^2-1\}$, where $j+1+pf_X([i])$ and $j+1+pf_Y([i])$ are considered module $p^2$. This implies that 
$$j+1+pf_X([i])\equiv j+1+pf_Y([i])\; (mod\; p^2) $$
for every $i,j\in \{0,...,p^2-1\}$ and hence $f_X=f_Y$.\\
The converse is an easy calculation.
\end{proof}

Using the previous results, a complete classification of indecomposable cycle sets of size $p^2$ and cyclic permutation group is obtained.

\begin{theor}
For every prime number $p$ there are, up to isomorphism, $p$ indecomposable cycle sets of cardinality $p^2$ and cyclic permutation group: $p-1$ having level $2$ and $1$ of level $1$. In particular, those having level $2$ are given (up to isomorphisms) by $X:=\{0,...,p^2-1\}$ and $\sigma_{i}(j):=t_1^{1+pf(i)}$ for every $i,j\in X$, where $f$ is one of the map of Lemma \ref{lemma2}.
\end{theor}

\begin{proof}
By Proposition \ref{dimoabel} and \cite[Theorem 3.7]{etingof1998set}, every indecomposable cycle set of size $p^2$ and cyclic permutation group is multipermutational of level at most $2$. Then the thesis follows by the previous proposition, Corollary \ref{prop3} and Lemma \ref{lemma2}.
\end{proof}

In the rest of this section we will show that the only indecomposable cycle set of cardinality $pq$ (where $p$ and $q$ are distinct prime numbers) having cyclic permutation group are those having level $1$. At first we exhibit a preliminary lemma.

\begin{lemma}\label{lemma1}
Let $X$ be a finite indecomposable cycle set of multipermutational level $2$ and with permutation group $\mathcal{G}(X)=<\alpha>$. If $\alpha^a$ and $\alpha^b$ are elements of $\sigma(X)$, then $a\equiv b\; (mod\; |\sigma(X)|)$.
\end{lemma}

\begin{proof}
Let $\alpha^a,\alpha^b\in \sigma(X)$ and $x_1,x_2\in X$ such that $\sigma_{x_1}:=\alpha^a$ and $\sigma_{x_2}:=\alpha^b$. Then $\sigma_{x_1}\cdot \sigma_z=\sigma_{x_2}\cdot \sigma_z$ for every $z\in X$.
So, $\sigma_{\alpha^a(z)}=\sigma_{\alpha^b(z)}$ for every $z\in X$ and hence $\sigma_{\alpha^a(\alpha^{-a}(z))}=\sigma_{\alpha^b(\alpha^{-a}(z))}$. Therefore $\sigma_{z}=\sigma_{\alpha^{b-a}(z)}$ for every $z\in X$, hence by Lemma \ref{retrgen} $a\equiv b\; (mod\; |\sigma(X)|)$.
\end{proof}

\begin{prop}\label{pdivq}
Let $X$ be a finite indecomposable cycle set of cardinality $pq$ and with cyclic permutation group $\mathcal{G}(X)=<\alpha>$, where $p$ and $q$ are prime numbers with $p<q$. Then $mpl(X)=1$.
\end{prop}

\begin{proof}
By Proposition \ref{dimoabel}, $X$ is a multipermutational cycle set and, by \cite[Theorem 3.7]{etingof1998set}, $mpl(X)\leq 2$ and $|\sigma(X)|\in \{p,q,1\}$.\\
If $mpl(X)=2$ and $|\sigma(X)|=q$, by Lemma \ref{lemma1} there exists $a\in \{1,\ldots pq-1\}$ such that $\sigma(X)\subseteq \{\alpha^{a+qt}| t\in \mathbb{N}\}$. Hence $|\sigma(X)|<q$, a contradiction.\\
If $|\sigma(X)|=p$ by Lemma \ref{lemma1} there exists $a\in \{1,\ldots pq-1\}$ such that $\sigma(X)$ is contained in $\{\alpha^{a+pt}| t\in \mathbb{N}\}$. Now, observe that $p\nmid a$, otherwise the permutation group $\mathcal{G}(X)$ is contained in $<\alpha^p> $, in contradiction with $|\mathcal{G}(X)|=pq$. Moreover, this implies that $id_X=\alpha^0\notin \sigma(X)$ and therefore $|<\alpha^q>\cap \sigma(X)|<p $. Let $x\in X$ and $t'\in \{1,...,q-1\}$ be such that $\sigma_x=\alpha^{a+pt'}$ and $q\nmid a+pt'$. Since $p\nmid a+pt'$ we obtain that $gcd(pq,a+pt')=1$ and hence $\mathcal{G}(X)=<\sigma_x>$. So, without loss of generality we can suppose that $\alpha=\sigma_x$.\\
Now, put $x_0:=x$ and $x_i:=\alpha^i(x)$ for every $i\in \{1,....,pq-1\}$. By Lemma \ref{retrgen} we have that $\sigma(X)=\{\sigma_{x_0},...,\sigma_{x_{p-1}}\}$, hence by \cite[Theorem 2.12]{etingof1998set}, it follows that $\sigma_{x_i}\cdotp \sigma_{x_j}=\sigma_{x_0}\cdotp \sigma_{x_j}=\sigma_{x_{j+1}}$ for every $x_i,x_j\in X$ (where $j+1$ is considered module $p$)\\
Hence, for every $i\in \{0,....,pq-1\}$ there exist $t_i\in \{0,...,q-1 \}$ such that $\sigma_{x_i}=\alpha^{1+pt_i}$, $t_i\neq t_{i+j}$ for every $j\in\{1,...,p-1\}$ and, by Lemma \ref{retrgen}, $t_i=t_{i+p}$. Now, since $\sigma_{x_i\cdotp x_j}\sigma_{x_i}=\sigma_{x_j\cdotp x_i}\sigma_{x_j}$ for every $i,j\in \{0,...,pq-1\}$ we have that 
$$2+pt_{j+1}+pt_i \equiv 2+pt_{i+1}+pt_j \;(mod\; pq)$$
which implies
$$t_{j+1}+t_i\equiv t_{i+1}+t_j\;(mod\; q)$$
for every $i,j\in \{0,...,pq-1\}$. Since $t_0=0$, we obtain that there exists $k\in \{1,...,q-1\}$ such that $t_i=\bar{ki}$, where $\bar{ki} $ is the rest of $ki$ in the division by $q$, for every $i\in\{1,...,p-1\}$. In this way, $t_p=t_0=0$ and 
$$-k(p-1)\equiv t_p-t_{p-1} \equiv t_1\equiv k\; (mod\; q),$$ 
hence $q\mid kp$, a contradiction. Therefore the thesis follows.
\end{proof}

\section{Indecomposable cycle sets of cardinality $pq$ and abelian permutation group}

In this section we focus on indecomposable cycle sets with a permutation group that is abelian non-cyclic. In the main result we will show that there is, up to isomorphism, only one indecomposable cycle set of size $p^2$, multipermutational level equal to $2$ and with abelian non-cyclic permutation group. This fact allows us to give a complete classification of indecomposable cycle sets of cardinality $pq$ and abelian permutation group.\\

At first, we show a construction of indecomposable cycle set of size $p^2$.

\begin{prop}
Let $(\mathbb{Z}/p\mathbb{Z}\times \mathbb{Z}/p\mathbb{Z},\cdotp_{\alpha})$ be the algebraic structure given by 
$$(a,i)\cdotp_{\alpha} (b,j):=(b+1,\alpha^{a}(j))$$
for all $(a,i),(b,j)\in \mathbb{Z}/p\mathbb{Z}\times \mathbb{Z}/p\mathbb{Z}$, where $\alpha$ is a $p-$cycle of $Sym(\mathbb{Z}/p\mathbb{Z})$. Then $(\mathbb{Z}/p\mathbb{Z}\times \mathbb{Z}/p\mathbb{Z},\cdotp_{\alpha})$ is an indecomposable cycle set having permutation group isomorphic to $\mathbb{Z}/p\mathbb{Z}\times \mathbb{Z}/p\mathbb{Z}$. Moreover, if $\alpha'$ is an other $p-$cycle of $Sym(\mathbb{Z}/p\mathbb{Z})$, we have that $(\mathbb{Z}/p\mathbb{Z}\times \mathbb{Z}/p\mathbb{Z},\cdotp_{\alpha})\cong(\mathbb{Z}/p\mathbb{Z}\times \mathbb{Z}/p\mathbb{Z},\cdotp_{\alpha'})$
\end{prop}

\begin{proof}
Verifying that $(\mathbb{Z}/p\mathbb{Z}\times \mathbb{Z}/p\mathbb{Z},\cdotp_{\alpha})$ is an indecomposable  cycle set having the permutation group isomorphic to $\mathbb{Z}/p\mathbb{Z}\times \mathbb{Z}/p\mathbb{Z}$ is a long but easy calculation.\\
Now, let $\alpha'$ be a $p-$cycle of $Sym(\mathbb{Z}/p\mathbb{Z})$. Then there exist an element $f\in Sym(\mathbb{Z}/p\mathbb{Z})$ such that $\alpha=f^{-1}\alpha'f$.\\
Let $\psi:(\mathbb{Z}/p\mathbb{Z}\times \mathbb{Z}/p\mathbb{Z},\cdotp_{\alpha}) \to (\mathbb{Z}/p\mathbb{Z}\times \mathbb{Z}/p\mathbb{Z},\cdotp_{\alpha'})$ be given by
$$(a,i)\mapsto (a,f(i))$$
for every $(a,i)\in (\mathbb{Z}/p\mathbb{Z}\times \mathbb{Z}/p\mathbb{Z},\cdotp_{\alpha})$. Then 
$$\psi((a,i)\cdotp (b,j))=\psi(b+1,\alpha^a(j))=(b+1,f(\alpha^a(j)))=$$
$$(b+1,\alpha'^a(f(j)))=(a,f(i))\cdotp (b,f(j))=\psi(a,i)\cdotp \psi(b,j) $$
for all $(a,i),(b,j)\in \mathbb{Z}/p\mathbb{Z}\times \mathbb{Z}/p\mathbb{Z}$. Since $\psi$ is bijective, the thesis follows.
\end{proof}

In the following result we will prove that, up to isomorphism, the unique indecomposable cycle set of size $p^2$ and permutation group isomorphic to $\mathbb{Z}/p\mathbb{Z}\times \mathbb{Z}/p\mathbb{Z}$ is the one obtained in the previous proposition.

\begin{prop}\label{uniqueinde}
Let $I$ be an indecomposable cycle set of size $p^2$, level $2$ and with permutation group $\mathcal{G}(I)$ isomorphic to $\mathbb{Z}/p\mathbb{Z}\times \mathbb{Z}/p\mathbb{Z}$. Then $I$ is isomorphic to the cycle set $(\mathbb{Z}/p\mathbb{Z}\times \mathbb{Z}/p\mathbb{Z},\cdotp_{\alpha})$  of the previous proposition.
\end{prop}

\begin{proof}
By \cite[Theorem 2.12]{vendramin2016extensions} and \cite[Proposition 2]{cacsp2018}, $I$ is isomorphic to a dynamical extension on $X\times_{\beta} S$, where $X:=\{x_1,...,x_p\}$ is a cycle set isomorphic to $\sigma(I)$, $S:=\{s_1,...,s_p\}$ and $\beta$ is a dynamical cocycle. Without loss of generality, by \cite[Theorem 2.12]{etingof1998set} we can suppose that $x_i\cdotp x_j=x_{\theta(j)}$ for every $x_i,x_j\in X$, where $\theta$ is the permutation of $Sym(\{1,...,p\})$ given by $\theta:=(1,...,p)$.\\
If $x\in X$, by \cite[Lemma 6]{cacsp2018} the subgroup $H_x=\{h|h\in \mathcal{G}(I), h(\{x\}\times S)\subseteq \{x\}\times S\}$ is transitive on $\{x\}\times S$ and since $\mathcal{G}(I)$ is abelian, by \cite[Lemma 5]{cacsp2018} we have that $H_x$ is transitive on $\{y\}\times S$ for every $y\in X$. Moreover, since $|\mathcal{G}(I)|=p^2$, necessarily $|H_x|=p$, hence $H_x$ is generated by a permutation $\gamma$ transitive on the subset $L_y:=\{(y,s)|s\in S\}$ for every $y\in X$, therefore 
$$\gamma:=((x_1,s_{k_{1,1}})...(x_1,s_{k_{1,p}}))...((x_p,s_{k_{p,1}})...(x_p,s_{k_{p,p}})).$$
Now, let $\eta$ be a left multiplication of $X\times_{\beta} S$. Then $\eta$ has order $p$ and, by the definition of the left multiplications in $X$, it acts as a $p$-cycle on the set $\Delta:=\{\{x\}\times S\mid x\in S\}$, hence 
$$\eta=((x_{1},s_{l_{1,1}})...(x_{p},s_{l_{1,p}}))...((x_{1},s_{l_{p,1}})...(x_{p},s_{l_{p,p}})).$$
It is not restrictive (after renaming the elements of $S$) supposing that $l_{i,j}=i$ for every $i,j\in\{1,...,p\}$, hence
$$\eta=((x_1,s_1)...(x_p,s_{1}))...((x_1,s_{p})...(x_p,s_{p})).$$
Now, let $\bar{k}_j\in Sym(\{1,...,p\})$ be given by $\bar{k}_j=(k_{j,1}...k_{j,p})$ for every $j\in \{1,...,p\}$. Then
$$\gamma \eta \gamma^{-1}= ((x_1,s_{\bar{k}_1(1)})...(x_p,s_{\bar{k}_p(1)}))...((x_1,s_{\bar{k}_1(p)})...(x_p,s_{\bar{k}_p(p)}))$$
and, since $\mathcal{G}(I)$ is abelian necessarily $\gamma \eta \gamma^{-1}=\eta$, therefore $\bar{k}_{j_1}=\bar{k}_{j_2}$ for every $j_1,j_2\in \{1,...,p\}$. So
$$\gamma:=((x_1,s_{k_{1,1}})...(x_1,s_{k_{1,p}}))...((x_p,s_{k_{1,1}})...(x_p,s_{k_{1,p}}))$$
and $\mathcal{G}(I)=<\eta>\times H_x=<\eta>\times <\gamma>$.\\
In this case, every left multiplication of the cycle set is in the form $\eta^t\gamma^u$ for some $t,u\in \{0,...,p-1\}$. By the definition of the left multiplication on $X$ we have that $\sigma_{(x_i,s_i)}=\eta\gamma^{n_i}$ for some $n_i\in \{0,...,p-1\}$ for every $i\in\{1,...,p\}$. We have to determine the numbers $n_i$.
Now, since $X\times S$ is a cycle set, equality (\ref{cicloide}) holds, and this implies that
\begin{equation}\label{equu}
n_{j+1}-n_j\equiv n_{i+1}-n_i \quad (mod\;p)
\end{equation}
for every $i,j\in \{1,...,p\}$ (where $i+1$ and $j+1$ are considered module $p$).\\
Since $|<\gamma>|=|X|=|\sigma(X)|=p$, there exist a unique $q\in \{1,...,p\}$ such that $\sigma_{(x_q,s)}=\eta\gamma^0=\eta$. After translating the variables $\{x_i\}_{i\in\{1,...,p\}}$, we can suppose that $q=p$. Therefore if $\sigma_{(x_1,s)}=\eta\gamma^{n_1} $ for some $n_1\in\{1,...,p-1 \}$, we obtain that 
$$\sigma_{(x_v,s_w)}=\eta\gamma^{n_1 v}$$
for every $v,w\in\{1,...,p\}$, hence the thesis.

\end{proof}

Finally, we can give the classification of indecomposable cycle sets of order $pq$ and with abelian permutation group.

\begin{theor}
Let $p$ and $q$ be prime numbers not necessarily distinct. Then the indecomposable cycle sets of order $pq$ and with abelian permutation group are, up to isomorphism, the following:
\begin{itemize}
\item[$\bullet$] Case $p\neq q$
\begin{itemize}
\item[a)] $(\mathbb{Z}/pq\mathbb{Z},\cdotp)$ where $i\cdotp j=j+1$ for all $i,j\in \mathbb{Z}/pq\mathbb{Z}$. In this case, $\mathcal{G}(\mathbb{Z}/pq\mathbb{Z})\cong (\mathbb{Z}/pq\mathbb{Z},+)$.
\end{itemize}
\item[$\bullet$] Case $p=q$ 
\begin{itemize}
\item[a)] $(\mathbb{Z}/p^2\mathbb{Z},\cdotp)$ where $i\cdotp j=j+1$ for all $i,j\in \mathbb{Z}/p^2\mathbb{Z}$. In this case, $\mathcal{G}(\mathbb{Z}/p^2\mathbb{Z})\cong (\mathbb{Z}/p^2\mathbb{Z},+)$.
\item[b)] $(\mathbb{Z}/p^2\mathbb{Z},\cdotp_f)$ where $i\cdotp j=j+1+pf(i)$ for all $i,j\in \mathbb{Z}/p^2\mathbb{Z}$ and $f:\mathbb{Z}/p\mathbb{Z}\longrightarrow \{0,...,p-1\}$ is a function as in Lemma \ref{lemma2}. In this case, $\mathcal{G}(\mathbb{Z}/p^2\mathbb{Z})\cong (\mathbb{Z}/p^2\mathbb{Z},+)$. Moreover, there are $p-1$ non-isomorphic cycle sets of this type.
\item[c)] $(\mathbb{Z}/p\mathbb{Z}\times \mathbb{Z}/p\mathbb{Z},\cdotp)$ where $(i,s)\cdotp (j,t)=(j+1,t+i)$ for all $i,j,s,t\in \mathbb{Z}/p\mathbb{Z}$. In this case, $\mathcal{G}(\mathbb{Z}/p\mathbb{Z}\times \mathbb{Z}/p\mathbb{Z})\cong (\mathbb{Z}/p\mathbb{Z}\times \mathbb{Z}/p\mathbb{Z},+)$.
\end{itemize}
\end{itemize}
\end{theor}

\begin{proof}
If an indecomposable cycle set $I$ has abelian permutation group $\mathcal{G}(I)$, by the transitivity of $\mathcal{G}(I)$ on $I$ it follows that $|\mathcal{G}(I)|=|I|$.\\
If $p\neq q$ then $\mathcal{G}(I)$ is isomorphic to $(\mathbb{Z}/pq\mathbb{Z},+)$ (the unique abelian group of size $pq$) while if $p=q$ we obtain that $\mathcal{G}(I)$ is isomorphic to $(\mathbb{Z}/p^2\mathbb{Z},+)$ or $(\mathbb{Z}/p\mathbb{Z}\times \mathbb{Z}/p\mathbb{Z},+)$.\\
Since by Proposition \ref{dimoabel} every indecomposable finite cycle set with abelian permutation group is multipermutational and by \cite[Theorem 3.7]{etingof1998set} an indecomposable cycle set of order $pq$ (where $p$ and $q$ are prime numbers not necessarily distinct) has multipermutational level at most $2$, the thesis follows by Corollary \ref{prop3}, Propositions \ref{propoiso}, \ref{pdivq} and \ref{uniqueinde}.
\end{proof}

Actually, by computer calculations, all the indecomposable cycle sets of size at most $8$ are known.
As final comment, referring to \cite[Problem $4$]{vendramin2018skew}, we remark that the previous result allows us to provide very easily all the indecomposable cycle sets having abelian permutation group and size $9$.

\section{Acknowledgments}

The authors would like to thank F. Catino for his comments and suggestions on this topic.

\section*{References}

\bibliographystyle{elsart-num-sort}
\bibliography{Bibliography}

\end{document}